\documentclass[a4paper,11pt,reqno]{amsart}

\usepackage[english]{babel}
\usepackage{amssymb}
\usepackage{mathtools}
\usepackage{tikz-cd}
\usepackage{xcolor}

\definecolor{darkgreen}{rgb}{0,0.50,0}
\definecolor{darkred}{rgb}{0.55,0,0}
\definecolor{darkblue}{rgb}{0,0,0.6}

\usepackage[
colorlinks=true,
citecolor=darkgreen,
linkcolor=darkred,
urlcolor=darkblue
]{hyperref}

\newtheorem*{thmB'}{Theorem B$'$}

\newtheorem{introthm}{Theorem}

\newtheorem*{thm*}{Theorem}

\newtheorem{theorem}{Theorem}[section]
\newtheorem{cor}{Corollary}[section]
\newtheorem{proposition}{Proposition}[section]
\newtheorem{lemma}{Lemma}[section]

\theoremstyle{definition}
\newtheorem{definition}{Definition}[section]
\newtheorem*{conv}{Convention}
\newtheorem{rem}{Remark}[section]

\makeatletter
\let\c@theorem\c@thm
\let\c@cor\c@thm
\let\c@proposition\c@thm
\let\c@lemma\c@thm
\let\c@definition\c@thm
\let\c@rem\c@thm
\makeatother

\counterwithout{equation}{section}

\newcommand{\ev}{\mathrm{ev}}
\newcommand{\pr}{\mathrm{pr}}
\newcommand{\gr}{\mathrm{gr}}
\newcommand{\Sym}{\mathrm{Sym}}
\newcommand{\E}{\mathbb{E}}
\newcommand{\T}{\mathcal{T}}
\newcommand{\Sign}{\mathrm{Sign}}
\newcommand{\p}{\mathcal{P}}
\newcommand{\U}{\mathcal{U}}

\newcommand{\n}{\mathbf{n}}
\newcommand{\nf}{\mathbf{[n(f)]}}
\newcommand{\bp}{\mathbf{p}}

\newcommand{\res}{\mathrm{Res}}
\newcommand{\diag}{\mathrm{diag}}

\newcommand{\Xn}{M(n)}
\newcommand{\Xnp}{M(n+1)}
\newcommand{\Xinfty}{M(\infty)}
\newcommand{\Minfty}{M(\infty)}

\newcommand{\dr}{\Delta_r}
\newcommand{\mr}{\mu_r}

\newcommand{\id}{\mathrm{id}}
\newcommand{\Hom}{\textup{Hom}}

\newcommand{\R}{\mathbb R}
\newcommand{\Z}{\mathbb Z}
\newcommand{\Q}{\mathbb Q}
\newcommand{\C}{\mathbb C}

\def\={\;=\;}  \def\+{\,+\,}   \def\thin{\hskip 1pt}
\def\be{\begin{equation} }  \def\ee{\end{equation}}

\newlength{\storeparskip}
\author[J. Semikina]{Julia Semikina}
\address{University of Lille, UMR 8524 - Paul Painlevé Mathematics Laboratory, F-59000 Lille, France}
\email{iuliia.semikina@univ-lille.fr}

\author[D. Zagier]{Don Zagier}
\address{\shortstack[l]{%
		Max Planck Institute for Mathematics, 53111 Bonn, Germany \textit{and}\\
		International Centre for Theoretical Physics, 34014 Trieste, Italy
}}
\email{don.zagier@mpim-bonn.mpg.de}

\date{}

\title{$L\thin$-classes of symmetric products: a homological formula}

\begin{document}
	
	\begin{abstract} 
If $M$ is a closed oriented manifold of even dimension, then its $n$th symmetric product $M(n)=M^n/S_n$
is a rational homology manifold for any~$n \geq 0$ and hence has a canonically defined $L$-class 
 in~$H^*(M(n);\Q)$. A complicated formula for these classes was given many years ago in the second author's thesis~\cite{ZagierThesis}. We obtain a simpler result by studying the {\it dual} classes  in homology, showing
that their generating function (after embedding each $H_*(M(n); \Q)$ into the Hopf algebra $H_*(\Minfty; \Q)\thin$)
is the product of an elementary factor depending only on the Euler characteristic of~$M$ and the exponential 
of a series consisting of odd Adams twists of the Poincar\'e dual of the $L$-class of~$M$. We also use this to obtain 
both a simpler proof and a simpler expression for the cohomological $L$-classes than the previous one. 
\end{abstract}

\maketitle
 
\vspace{-2ex}

\begingroup%
\setlength{\parskip}{\storeparskip}
\setcounter{tocdepth}{1}
\tableofcontents
\endgroup%

\setcounter{section}{0}

\section*{Introduction}

The Hirzebruch signature theorem says that
$$   \Sign(X) \= \langle L(X), [X]\rangle $$
for any closed oriented differentiable manifold~$X$, where  $L(X)\in H^*(X;\Q)$, the
{\it $L$-class} of~$X$, is determined by, and determines, the rational Pontryagin class of~$X$.
Using this, Thom~\cite{thom1958} extended the definition of the $L$-class to triangulated manifolds and Milnor~\cite{milnor1974characteristic} to oriented rational homology manifolds (the rational Pontryagin 
classes come along with it, though we will not use them).
In~\cite{zagier1972} and~\cite{ZagierThesis} the second author calculated this
$L$-class first for quotients of an oriented manifold by a finite group action and then in detail 
for the case when $X$ is the $n$-th symmetric product  $M(n)=M^n/S_n$ of a  closed 
orientable differentiable manifold~$M$ of even dimension~$2s$.  (The evenness is necessary to make
$X$~orientable).  The final formula had the attractive shape
 \be\label{eq: main Don's thesis thm}  
     L(M(n)) \ = \ j_n^* \big(Q_s(\eta)^{n+1} \thin G \big)\thin,    \ee
where $j_n$ denotes the canonical inclusion of~$M(n)$ into $M(\infty)$ (after fixing a base point), $\eta$ is a
canonically defined class in~$ H^{2s}(M(\infty),\Q)$, $Q_s(t)\in\Q[[t]]$ is a fairly simple power series depending 
only on~$s$ (given in~\eqref{DefQs} and~\eqref{Defgs} below), and $G \in H^*(\Minfty)$ is independent of $n$.  However,
the expression for $G$ given in~\cite{ZagierThesis} was very complicated and the proof was also quite intricate. 
In this paper we will give a  more elegant formula for the class~$G$ and also a much simpler proof by working in homology rather than cohomology. 

Specifically, define $l(X)\in H_*(X;\Q)$, the {\it homological $l$-class} of~$X$, for any rational homology 
manifold~$X$ as the Poincar\'e dual of the class~$L(X)$.
The classes $l(M(n))$ lie in different groups, so to make a generating series we must take their push-forwards
via~$j_n$ to $H_*(\Minfty)$, which has the structure of  a $\Q$-algebra via the Pontryagin product induced by 
the monoid structure of~$\Minfty\thin$. Our main result is then:
 \begin{introthm} \label{main theorem: homological formula for l classes}
     Let $M$ be a compact oriented smooth manifold of even dimension. Then     
\be  \label{lSeries}
 \sum_{n=0}^{\infty} \, (j_n)_* \, l(\Xn) \cdot t^n \= \bigl(1-t^2\bigr)^{-e(M)/2}  \thin\cdot\thin
e^{\Psi(l(M))} \ \in \ H_*(\Minfty)[[t]]\thin .       \ee
Here $e(M)$ denotes the Euler characteristic of $M$ and $\Psi\thin:\thin H_*(M)\to H_*(\Minfty)[[t]]$
is a degree-preserving map encoding the twists by odd Adams operations.
 \end{introthm} 
Note that the degree~0 part of this theorem recovers the generating series
\be \label{eq: gen series for signature}
    \sum_{n=0}^{\infty} \; \Sign ( M(n)) \, t^n \= (1-t^2)^{-e(M)/2} \, \Big( \frac{1+t}{1-t}\Big)^{\Sign(M)/2}\,,
\ee
given by Hirzebruch~\cite{HirzGenSer}, and also that it implies that the $L$-classes of all
symmetric products of~$M$ are completely determined by the cohomology ring (with its Poincar\'e duality
structure) and Hirzebruch $L$-class of~$M$ itself. 

We will apply Theorem~\ref{main theorem: homological formula for l classes} to give a new short proof of the 
original formula~\eqref{eq: main Don's thesis thm}\thin:
 \begin{introthm} \label{thmB}
Let $M(n)$ be the $n$-th symmetric product of $M$. Then
  \begin{equation*} 
     L(M(n)) \ = \ j_n^*\bigr(Q_s(\eta)^{n+1} \thin G\bigr),
  \end{equation*}
where $s=\frac{1}{2} \dim(M)$, $Q_s(t)\in\Q[[t]]$ is the power series 
\be\label{DefQs} Q_s(t)\=\frac{t}{f_s(t)}
 \= 1 \+ \frac 1{3^s}\thin t^2 \+ \Bigl(\frac 1{5^s}-\frac 2{9^s}\Bigr)\thin t^4 \+ \cdots \ee
with $f_s(t)$ defined as the inverse power series of the odd polylogarithm
\be\label{Defgs} 
g_s(t)\ = \ \sum_{\substack{r\geq 1\\ r  \mathrm{~odd}}} \frac{t^r}{r^s} \ = \ t+\frac{t^3}{3^s}+\frac{t^5}{5^s}+\cdots\,, 
\ee
and $G$ is a class in $H^*(\Minfty)$ independent of $n$.
\end{introthm}

Finally, we will derive an explicit closed expression for the factor~$G$, considerably simpler than the one presented 
in~\cite{ZagierThesis}. To formulate it, we use the structure of the cohomology of the
infinite symmetric product, established in Section~\ref{Sect2}. Namely, the algebra $H^*(M(\infty);\Q)$
is free graded-commutative on primitive classes $p_0 = \eta,\, p_1, \ldots, p_{b-1}$ 
which are degree-preserving images of an additive basis for $H_{>0}(M)$, and has a canonical
degree-preserving automorphism~$\Phi$ which is the base-change map relating the monomials in the
classes $p_i$ to an additive basis given by symmetrizations (see Section~\ref{Sect2} for details).
\begin{introthm}  \label{thmC}
The factor $G \in H^*(M(\infty))$ of Theorem~\ref{thmB} is given by
\[
    G \,=\, \frac{f'_s(\eta)}{(1-f_s(\eta)^2)^{e(M)/2}} \; \Phi(e^H)
\]
with $H \in \Q[[\eta]][p_1, \ldots, p_{b-1}]$ defined as
 \begin{equation}
        H=\sum_{i=0}^{s} \sum_{k=0}^{2i} \frac{1}{k!}\Big(\frac{f_s(\eta)}{\eta} \Big)^k  g^{(k)}_{i+1}(f_s(\eta)) \, \varepsilon^k(l_i),
     \end{equation}
     where $g_{i+1}^{(k)}$ denotes the $k$-th derivative of the odd polylogarithm $g_{i+1}$, $l_i \in H_{2i}(M)$ is the degree~$2i$ 
component of the homological $l$-class of $M$, and the map $\varepsilon^k\thin:\thin H_*(M)\to H^*(\Minfty)$ 
encodes the evaluations of $k$-fold products of basis elements and is introduced in Section~\ref{Sect6}.

\end{introthm}

\subsection*{Earlier results.}  
Macdonald~\cite{MacCurves} gave an expression for the Chern classes, hence implicitly also for the Pontryagin classes
and the $L$-class, of the $n$-th symmetric product of a complex curve (in which case
$M(n)$ is smooth for all~$n$ and the Thom–Milnor generalization of the Hirzebruch $L$-class is not needed).
We should  also mention that equation~\eqref{lSeries} was proved in~\cite{cappell2017characteristic}
(Theorem 1.5) in the special case when $M$ is a smooth complex projective variety (the authors also treat 
arbitrary complex quasi-projective varieties, in which case their result is not a special case of 
Theorem~\ref{main theorem: homological formula for l classes}), but their proof uses
the machinery of algebraic geometry and hence does not apply to the differentiable case.
However, we only learned about this while writing this paper, which is based on a preliminary version written many
years earlier by the second author.

\subsection*{Overview.} The paper is organized as follows. Section~\ref{Sect1} collects preliminaries on $L$-classes 
of rational homology manifolds as well as equivariant $L$-classes. Section~\ref{Sect2} establishes the properties and
the Hopf algebra structure of the (co)homology of the infinite symmetric product. Sections~\ref{Sect3} and~\ref{Sect4}
contain the proof of Theorem~\ref{main theorem: homological formula for l classes}. Section~\ref{Sect5} gives the
new proof of the cohomological formula for~$L(M(n))$, and Section~\ref{Sect6} the closed formula for the factor~$G$.
\begin{conv} 
The manifolds considered in this paper will always be connected compact oriented smooth unless stated otherwise. 
All (co)homology groups are taken with coefficients in a field of characteristic zero (mostly $\Q$ or $\C$). We mostly
omit the coefficients and coefficient homomorphisms, so that, for example, we will multiply classes in $H^*(X; \Q)$ 
by elements of $H^*(X; \C)$ without explicit comment. Throughout the paper, the lower-case $l$ will denote the 
homological $l$-class, while the capital~$L$ is reserved for the cohomological Hirzebruch $L$-class.
\end{conv}

\subsection*{Acknowledgements.}
The first author acknowledges the support of the CDP C2EMPI, as well as the French State under the France-2030 
programme, the University of Lille, the Initiative of Excellence of the University of Lille, the European Metropolis 
of Lille for their funding and support of the R-CDP-24-004-C2EMPI project.
\section{Preliminaries} \label{Sect1}

\subsection*{The signature theorem and the $L$-class}
Recall that the signature $\Sign(X)$ of a closed oriented smooth manifold $X$ is defined as the signature of the
 intersection form  on~$H^{2k}(X)$ if $\dim X=4k$ and zero otherwise. The celebrated result of 
Hirzebruch~\cite{hirzebruch1966topological} states that the signature can be expressed as a linear combination
of Pontryagin numbers via the $L$-class. The Hirzebruch $L$-class
\[
L(X)\, =\, L(TX)\, =\; \prod_{j} \frac{x_j}{\tanh(x_j)}
\]
is the multiplicative sequence in the Pontryagin classes of $TX$ associated to the power series $x/\tanh(x)$. 
Here, $x_j^2$ are the formal  Pontryagin roots defined by $p(TX)=\prod_j (1+x_j^2)$. We refer the reader 
to~\cite{hirzebruch1966topological} and~\cite{milnor1974characteristic}  for a more detailed exposition. 
The signature theorem states that
\[
\Sign(X) \, =\, \langle L(X), [X] \rangle. 
\]
For any $4d$-dimensional submanifold $A \subset X$ with trivial normal bundle the $L$-class of $A$ is the
 restriction $i^* L(X)$ (where $i$ denotes the inclusion) and hence 
\[
\langle L(X), i_*[A] \rangle \, =\, \Sign(A).
\]
The observation of Milnor~\cite{milnor1974characteristic}, extending earlier work of Thom~\cite{thom1958},
is that for a rational homology manifold $Y$ one
can define an analogue of a ``submanifold with trivial normal bundle'' (basically the inverse image of a generic
point of a map from $Y$ to a sphere) and by the work of Serre~\cite{serre1953groupes} there are enough of 
these special subspaces to represent all of $H_*(Y; \Q)$.  A \textit{rational homology manifold} of dimension~$n$
 is a locally finite simplicial complex~$Y$ such that for all $y \in Y$
\[
H_i(Y, Y-{y}; \Q) \cong H_i(\R^n, \R^n-\{0\}; \Q) \cong \begin{cases} \Q &\text{if~} i=n,\\ 0 &\text{if~} i \neq n. \end{cases}
\]
Such a compact oriented space~$Y$ has a fundamental class $[Y]$ and satisfies Poincar{\'e} duality over~$\Q$.
To define the~$L$-class for a rational homology manifold~$Y$, we use Serre's theorem~\cite{serre1953groupes}, 
which states that for $n=\dim Y \leq 2i-2$ there is a rational isomorphism of abelian groups: 
\begin{align*}
    \pi^i(Y) \coloneqq [Y, S^i] \ &\cong_{\Q} \ H^i(Y)\\
    [f] &\longmapsto f^* (u),
\end{align*}
where $u \in H^i(S^i)$ is the standard generator. Given a simplicial map $f \colon Y \to S^i$, the preimage $f^{-1}(p)$
of almost every point is again an oriented rational homology manifold, whose signature only depends on~$[f]$. This signature
vanishes unless the codimension~$n-i$ is equal to~$4k$ for some~$k$, and in that case defines a homomorphism 
$\pi^{n-4k}(Y) \to \Z$ which, combined with Serre's theorem, yields a unique class $L_k \in H^{4k}(Y)$ such that
\[
\langle L_k \cup f^*(u), [Y] \rangle \ = \ \Sign(f^{-1}(p))\thin.
\]
The restriction on $i$ can be removed by applying the above procedure to $Y \times S^N$ for large~$N$ instead. 
The $L$-class is then defined as the sum of these classes:
\[
L(Y) \, \coloneqq \, \sum_k L_k\ \in \ H^{4*}(Y; \Q)\thin.
\]
For a smooth manifold $Y$  it agrees with the Hirzebruch $L$-class. This definition can be stated more elegantly for 
the homological $l$-class $l(Y)=L(Y)\cap[Y]\thin$: 
\[
\langle f^*(u),  l(Y) \rangle \ = \ \Sign(f^{-1}(p)),
\]
so the homological $l$-class is actually the primary object to consider.
\subsection*{Duality and Umkehr map}
Recall from the conventions that all (co)homology groups have rational coefficients unless stated otherwise. For a 
compact oriented rational homology manifold~$Y$ of dimension $n$ the Poincar{\'e} duality isomorphism is given 
by capping with the fundamental class~$[Y]$:
\[
\mathrm{PD}_{Y}=\ -\cap[Y] \ \colon \ H^{k}(Y) \xrightarrow{\cong} H_{n-k}(Y).
\]
For an inclusion~$i \colon F \to Y$ of compact oriented rational homology manifolds of codimension~$d$ the Umkehr (
or Gysin) homomorphism $i_!$ is defined via the following commutative diagram:
  \[
     \begin{tikzcd}
H^{a-d}(F) \arrow[r, dashed, "i_{!}"] 
\arrow[d, "{\scriptstyle \cong}", "\mathrm{PD}_{F}"'] & H^{a}(Y) 
\arrow[d,"\mathrm{PD}_{Y}", "{\scriptstyle \cong}"']\\
H_{n-a}(F) \arrow{r}{i_*}  & H_{n-a}(Y)
	\end{tikzcd}
   \]
for any $a$ between 0 and~$n$. In other words, it is the composite homomorphism
 \[
 i_! \coloneqq  PD^{-1}_Y\circ i_* \circ PD_F \ \colon \ H^{a-d}(F) \to H^{a}(Y).
 \]
\subsection*{Equivariant $L$-class in cohomology and in homology} Let $X$ be a closed oriented smooth manifold with
an action of a finite group $G$ by orientation-preserving diffeomorphisms. Let $g \in G$ and suppose that the fixed-point
set $X^g$ is an oriented submanifold (which will be true in our case). The normal bundle $N^g$ of $X^g$ inherits a 
faithful action of $g$. At each point $x \in X$ the fiber $N^g_x$ can be decomposed as a sum of one-dimensional subspaces,
where $g$ acts as a reflection, and two-dimensional subspaces, where $g$ acts by rotation through the angle $\theta$. 
Rotation by $\theta$ and by $-\theta$ are equivalent, so we can assume $0 < \theta < \pi$. Such a decomposition of a fiber
extends to a splitting of the whole bundle $N^g$ over each component $X^g$:
\[
N^g \= N_{\pi}^g \oplus \bigoplus_{0 < \theta < \pi} N_{\theta}^g,
\] 
where~$N_{\pi}^g$ is a real bundle on which $g$ acts as~$-1$ and~$N_{\theta}^g$ is a complex bundle on which $g$ acts as~$e^{i\theta}$.

Let $\xi$ be a $U(q)$-bundle over a space $Y$ and $x_j$ its Chern roots, defined by the formal $q$-fold product $c(\xi)=\prod_j (1+x_j)$. Following~\cite{zagier1972}, for $\theta \in \R \smallsetminus 2\pi\Z$ set
 \begin{equation}
      L_{\theta}(\xi) \, \coloneqq \, \prod_j \coth\Big(x_j+\frac{i\theta}{2}\Big) \ \in \ H^*(Y;\Q(e^{i\theta})),
  \end{equation} 
  where, as before, this means expanding the symmetric expression as a power series in the $x_j$'s and substituting the
elementary symmetric polynomials with the corresponding Chern classes $c_k(\xi)$. Note that $L_{\theta}$ is not stable
as its degree-zero term is~$\coth(i\theta/2)^q$.
Similarly, for an $O(q)$-bundle~$\xi$ over~$Y$ set
 \begin{equation}
  L_{\pi} (\xi) \,  \coloneqq \,  e(\xi)L(\xi)^{-1} \in H^*(Y; \Q),
 \end{equation} 
 where $e(\xi) \in H^q(Y; \Q)$ is the Euler class and $L(\xi)$ is the Hirzebruch $L$-class of $\xi$.
 The class $L'(g, X)$ is now defined as the product of three components:
    \begin{equation} \label{eq: three factors of L'}
    L'(g, M)=L(X^g) \ L_{\pi}(N^g_{\pi}) \ \prod_{0<\theta<\pi} L_{\theta}(N^g_{\theta}) \ \in \ H^*(X^g; \C).
    \end{equation}  
Applying the Umkehr map for the inclusion $i^g \colon X^g \to X$ we obtain the $g$-th equivariant~$L$-class:
\begin{equation} \label{eq: equivariant L-class def}
    L(g, X) \ \coloneqq \ i^g_{\thin!} \thin L'(g, X) \ \in H^*(X; \C).
\end{equation}
One of the main results in~\cite{zagier1972}, used as the crucial tool to compute the~$L$-classes of symmetric products, 
is the following averaging formula.
\begin{thm*}
     Given a manifold $X$ with an action of a finite group $G$ as above, the $L$-class of the rational homology manifold~$X/G$
is the average of the equivariant $L$-classes of $X$: 
  \begin{equation} \label{eq: averaging formula for L(g, M)}
  \frac{1}{\deg \pi} \ \pi^* L(X/G) \ =\ \frac{1}{|G|} \; \sum_{g \in G} L(g, X),
   \end{equation}
where $\pi \colon X \to X/G$ is the projection map.
\end{thm*}
Following our philosophy that homological $l$-classes are the primary objects to study, let us define the homological counterparts of the $L'$-classes:
\[
l'(g, X) \coloneqq L'(g, X) \cap [X^g] \in H_*(X^g; \C)\thin,
\]
so that in view of the definition of $i^g_{\thin!}$ the equivariant $L$-class defined by~\eqref{eq: equivariant L-class def} is simply
the Poincar\'e dual of $i^g_{\thin*}\l'(g,X)$.
 For an intrinsic treatment of equivariant homological $l$-classes for singular spaces we refer the reader to~\cite[Section 9]{banagl2026equivariant}. Note the difference in notation: what we call $l'(g, X)$ is denoted by $L_*(g, X)$ in~\cite{banagl2026equivariant}.

\section{Homology and cohomology of infinite symmetric products} \label{Sect2}

\subsection*{Symmetric products} Let $M$ be a compact oriented smooth manifold of even dimension $2s$. In this case, the action of the symmetric group $S_n$ on $M^n$ by permutation of the coordinates is orientation-preserving: the interchange map $T \colon M \times M \to M \times M$ preserves the orientation as there is no extra sign coming from the graded-commutativity when the dimension of $M$ is even. We define the \textit{$n$-th symmetric product} of $M$ as the quotient space $M(n)=M^n/S_n$, also denoted in the literature by $M[n]$ or $\mathrm{SP}^n M$, and denote the projection map $M^n \to M(n)$ by $q_n$. Since $M(n)$ is the quotient of an oriented smooth manifold by a finite orientation-preserving action it is an oriented rational homology manifold. An element of $M(n)$ is given by an unordered $n$-tuple (a~multiset), which we still write as $(x_1, \ldots, x_n)$ when no ambiguity arises. If we choose a base-point $* \in M$, then there is a natural inclusion
\begin{align*}
    j_{n, n+1} \colon M(n) &\longrightarrow M(n+1)\\
   (x_1, \ldots, x_n)  &\mapsto (x_1, \ldots, x_n, *).
\end{align*}
The \textit{infinite symmetric product} is the colimit of the direct system of maps induced by the~$j_{n, n+1}$:
\[
\Minfty \ \coloneqq\ \lim_{\longrightarrow} \, M(n).
\]
We denote by $j_n \colon M(n) \to \Minfty$ the inclusion defined by $j_n(x_1, \ldots, x_n)=(x_1, \ldots, x_n, *,*, \ldots)$.

The map given by the union of multisets 
 \[
m_{n,n'} \colon M(n) \times M(n') \to M(n+n'), ~~\big( (x_1, \dots, x_n),(y_1, \ldots, y_{n'}) \big) \mapsto (x_1, \ldots, x_n, y_1, \ldots, y_{n'})
 \]
 induces a multiplication map on $\Minfty$, which we denote by the same letter $m$:
   \begin{equation} \label{multiplication map on Xinfty}
     m \colon \Minfty \times \Minfty \to \Minfty.
   \end{equation}
This turns $\Minfty$ into an abelian topological monoid with identity $(*, *, \ldots)$. In particular, $\Minfty$ is an $H$-space, which naturally equips its homology and cohomology with a Hopf algebra structure. We discuss this next.

Over a field, in our case $\Q$, the homological cross-product 
\[
\times \colon \bigoplus_{p+q=k} H_p(A) \otimes H_q(B) \xrightarrow{\cong} H_{k}(A \times B)
\]
is an isomorphism, by the Künneth theorem. 
 \begin{definition}
On homology, there is an associated \textit{Pontryagin product} given by the composite
    \[
H_p(\Minfty) \otimes H_q(\Minfty) \xrightarrow{\times} H_{p+q}(\Minfty \times \Minfty) \xrightarrow{m_*} H_{p+q}(\Minfty). 
    \]
For the Pontryagin product of two homology classes $x, y \in H_*(\Minfty)$ we will simply write $xy$, meaning    
  \[
 xy \, \coloneqq \,  m_*(x \times y).
  \]
    \end{definition}
Let $\delta \colon \Minfty \to \Minfty \times \Minfty$ denote the diagonal map. Composing $\delta_*$ with the Künneth isomorphism endows $H_*(\Minfty)$ with a coproduct, which together with the Pontryagin product makes $H_*(\Minfty)$ a connected, graded-commutative and cocommutative Hopf algebra over~$\Q$. Since $H_*(\Minfty)$ is of finite type, the dual $H^*(\Minfty)$ is also a connected graded-commutative and cocommutative Hopf algebra of finite type, with product the cup product and coproduct induced by $m^*$.

\subsection*{Adams operations.} Since $\Minfty$ is a topological monoid, we have the $r$-th power map:
\begin{equation} \label{eq:Adams operation on Hopf algebras}
    \phi_r \; \colon  \; \Minfty \to \Minfty, ~~~x \mapsto m^{(r)} (\underbrace {x,x, \ldots, x}_r),
\end{equation}
where $m^{(r)}$ is the $r$-th iterate of the multiplication map~\eqref{multiplication map on Xinfty}. The induced map $(\phi_r)_* \; \colon \; H_*(\Minfty) \to H_*(\Minfty)$ can be written as $(\phi_r)_*=m^{(r)}_* \circ \delta^{(r)}_*$, which is the \textit{classical Adams operation} (also called Sweedler powers or the characteristic operation) on the Hopf algebra $H_*(\Minfty)$ (see e.g.~\cite{cartier2007primer}).

Let
\[
\dr \, \colon \, M \longrightarrow \Minfty
\]
be the $r$-fold diagonal map given by
\[
x \longmapsto
( \underbrace {x, \ldots, x}_{r}, *, *, \ldots ).
\]
Equivalently, $\dr$ is the composition
\[
M
\xrightarrow{\ \operatorname{diag_r}\ }
M^r
\xrightarrow{\ q_r \ }
M(r)
\xhookrightarrow{\ j_r \ }
\Minfty .
\]
The $r$-th \textit{homological Adams operation}   (note that it is different from the operation on Hopf algebras above) $\psi_r$ is defined as the rescaling of the even-dimensional homogeneous components according to their degree (see~\cite{adams1963vector}):
\[
\psi_r : H_{2i}(\Xinfty) \xrightarrow{\cong} 
H_{2i}(\Xinfty), ~~~x \mapsto r^{-i} \, x.
\]
We then define
\[
\mr \, \colon \, H_{\mathrm{ev}}(M)
\longrightarrow
H_{\mathrm{ev}}(\Minfty)
\]
by first applying the push-forward induced by $\dr$ followed by the $r$-th homological Adams operation:
\begin{equation} \label{eq: mu_r definition}
	\mr \colon
	H_{\mathrm{ev}}(M)
	\xrightarrow{(\dr)*}
	H_{\mathrm{ev}}(\Minfty)
	\xrightarrow{\psi_r}
	H_{\mathrm{ev}}(\Minfty).
\end{equation}
We observe that $(\Delta_r)_*$ factors through the $r$-power map ($r$-th Hopf algebra Adams operation) 
from~\eqref{eq:Adams operation on Hopf algebras}: $(\Delta_r)_*\, =\, (\phi_r)_* \circ (j_1)_*$. Let us denote
by $\Psi_r$ the composition of the two Adams operations we have defined, one by taking powers (Hopf algebraic) and the other one homological (rescaling):
\[
\Psi_r \; \coloneqq \; \psi_r \circ (\phi_r)_* \colon H_{\ev}(\Minfty) \to H_{\ev}(\Minfty).
\]
We then define
\[
\Psi \colon H_{\mathrm{ev}}(M)
\longrightarrow
H_{\mathrm{ev}}(\Xinfty)\otimes \mathbb{Q}[[t]]
\]
by
\be \label{PsiSer}
\Psi(x)
\; \coloneqq \;
\sum_{\substack{r\geq 1\\ r  \mathrm{~odd}}}
\mr(x) \cdot \frac{t^r}{r} 
\; =\; \sum_{\substack{r\geq 1\\ r  \mathrm{~odd}}}
\psi_r  \big((\Delta_r)_*(x)\big) \cdot \frac{t^r}{r} \; =\;  \sum_{\substack{r\geq 1\\ r  \mathrm{~odd}}}
\Psi_r  \big( (j_1)_*(x) \big) \cdot \frac{t^r}{r}.
\ee
In particular, $\Psi$ is a degree-preserving map, where the formal variable $t$ is assigned degree zero, and it takes 
values in odd power series in $t$. This will be the universal series whose exponential computes the generating series 
of $l$-classes in Theorem~\ref{main theorem: homological formula for l classes}.
  \subsection*{Algebra structure on $H_*(\Minfty)$} Let $\{e_0, e_1, \ldots, e_b \}$ be a homogeneous additive basis 
of $H^*(M)$ with $e_i$ of degree $d_i$. We normalize so that $e_0=1$ is the unit class in degree $0$ and $e_b$ is the 
top dimensional class dual to $[M] \in H_{2s}(M)$, i.e. $\langle e_b, [M]\rangle =1$. Dually, in homology, $\{e^*_0, e^*_1, \ldots, e^*_b \}$
is an additive basis of $H_*(M)$ with $e^*_0=[*]$, the class of the chosen base point, and $e^*_b=[M]$. Set 
  \begin{equation} \label{eq: def of u_i}
      u_i \, \coloneqq \, (j_1)_*(e^*_i) \ \in \ H_{d_i}(\Minfty),
  \end{equation}
  where $j_1 \colon M=M(1) \to \Minfty$ is the inclusion. The top-dimensional class will be of particular importance 
and we give it a special name:
 \begin{equation}
 \alpha \, \coloneqq u_b \, = \, (j_1)_*[M] \ \in \ H_{2s}(\Minfty).
 \end{equation}
\begin{theorem} \label{thm: free hopf algebra}
The algebra $H_*(\Minfty)$ is free graded-commutative on $u_1, \ldots, u_b$: 
\[
H_*(\Minfty) \; \cong \; \Q\big[u_i \mid  i>0,\;d_i \textup{ even}\big] \otimes \Lambda \big(u_i \; |\; d_i \textup{ odd}\big).
\]

\end{theorem} 

\begin{proof}
Since we work over $\Q$ and $S_n$ is a finite group, the projection $q_n \colon M^n \to M(n)$
induces an isomorphism from the
coinvariants
\[
(q_n)_* \colon H_*(M^n)_{S_n} \xrightarrow{\; \cong \;} H_*(M(n)).
\] 
Furthermore, by the Künneth theorem we have
\[
H_*(M^n)_{S_n} \; \cong \; \big( H_*(M)^{\otimes n} \big)_{S_n} \= \Sym^n(H_*(M)).
\]
Here $\Sym^n(H_*(M))$, sometimes denoted $\Sym^n_{\gr}(H_*(M))$, is the graded-symmetric $n$th power
of the algebra $H_*(M)$, i.e., its usual $n$th symmetric power but with the obvious sign convention for products.
We endow $\Sym(H_*(M))=\bigoplus_{n} \Sym^n(H_*(M))$ with the product induced by concatenation. The identification
\be \label{EqSym}
\Sym^n(H_*(M))  \; \cong \; H_*(M(n)), ~~[x_1 \otimes\cdots \otimes x_n] \mapsto (q_n)_*\big(x_1 \times\cdots \times x_n \big)
\ee
respects the corresponding products, since $q_{n+k}=m_{n,k} \circ (q_n \times q_k)$, where $m$ is the multiplication defined in \eqref{multiplication map on Xinfty}. The inclusion map $j_{n, n+1} \, \colon \, M(n) \to M(n+1)$ is given by adjoining a 
base point $x \mapsto (x, *)$. Hence, under~\eqref{EqSym},  the induced map $(j_{n, n+1})_*$ is multiplication by $e^*_0=[*]$. 
In the colimit the element~$e^*_0$ becomes the multiplicative unit and therefore
\[
H_*(\Minfty) \; =\;  \varinjlim  H_*(M(n)) \; \cong  \; \Sym (\tilde{H}_*(M)).
\]
Finally, the basis elements $e^*_i \in \tilde{H}_*(M)$ correspond to $u_i=(j_1)_*(e^*_i)$, which finishes the proof.
\end{proof}
 
 \begin{lemma} \label{lem: [X(n)] in alpha} The pushforward of the fundamental class $[M(n)]$ along the inclusion $j_n \colon M(n) \to \Minfty$ is given by
    \[
    (j_n)_*[M(n)] \; =\; \frac{\alpha^n}{n!} \ \in \ H_*(\Minfty).
    \]
\end{lemma}
\begin{proof}
   Let $q_n \, \colon \, M^n \to M(n)$ be the quotient map. We have $j_n \circ q_n \; =\; m^{(n)} \circ (j_1)^{\times n}$, where $m^{(n)}$ is the iterated multiplication \eqref{multiplication map on Xinfty}. Using that $q_n$ has degree $n!$, we obtain $n! \, (j_n)_* [M(n)]  \,=\, \alpha^n.$
\end{proof}
 
 \subsection*{Additive and multiplicative bases for $H^*(\Xinfty)$} Let $f_0, \ldots, f_b$ be an additive basis of $H^*(M)$ Poincar{\'e} dual to the basis $e_0, \ldots, e_b$, meaning
 \[
 \langle e_i \cup f_j , [M]\rangle \; =\; \delta_{ij}.
 \]
 Hence each $f_i$ is homogeneous of degree $(2s-d_i)$, in particular $f_b=1$ and  $f_0$ is the cohomological fundamental 
class, with $\deg f_0=2s$. The description of an additive basis follows from the fact that the cohomology of $M(n)$ is 
identified with the invariants $H^*(M^n)^{S_n}$. We refer the reader to~\cite[Section 7]{ZagierThesis} for details. 
Write $\pi_k \colon M^n \to M$ for the projection onto the $k$-th factor. Given homogeneous elements $v_1, \ldots, v_n \in H^*(M)$,
define the \textit{symmetrization} of $v_1 \times\cdots \times v_n$ as
  \[
 \langle v_1, \ldots, v_n \rangle \ \coloneqq \ \sum_{\sigma \in S_n} \sigma^*(v_1 \times \cdots \times v_n) \ = \ \sum_{\sigma \in S_n} \pi^*_{\sigma(1)} (v_1) \ldots \pi^*_{\sigma(n)} (v_n)
  \]
and set
   \[
 \langle n_0(f_0) \ldots n_b(f_b) \rangle \ \coloneqq \ \langle \underbrace {f_0, \ldots, f_0}_{n_0 }, \underbrace{f_1, \ldots, f_1}_{n_1}, \ldots , \underbrace{f_b, \ldots, f_b}_{n_b} \rangle,
   \]
where $n_0, \ldots, n_b$ are integers satisfying $n_0+\ldots +n_b=n$ and $n_i \leq 1$ if $2s-d_i=\deg f_i$ is odd. The set 
of these elements forms an additive basis of $H^*(\Xn, \Q)$. 
Set
\[
\nf_{n} \; \coloneqq \;  [n_0(f_0) \ldots n_{b-1}(f_{b-1})]_n \ = \ \begin{cases} \frac{1}{n_{b}!} \, \langle n_0(f_0) \ldots n_b(f_b) \rangle \; & \text{if~} n \geq n_0+\cdots+n_{b-1}, \\
	0  \; & \text{if~} n < n_0+\cdots+n_{b-1},
	\end{cases}
\]
where $\n \coloneqq (n_0, \ldots, n_{b-1})$ and $n_b \coloneqq n-(n_0+\cdots+n_{b-1}).$
\begin{lemma} \label{lemma: j_n injectivity on homology}
Let $j_{n, n+1} \colon \Xn \to \Xnp$ be the inclusion. The restriction map $j^*_{n, n+1}$ is surjective with
\[
j_{n,n+1}^* \big( \nf_{n+1} \big) \=
    \nf_{n}.
\]
Dually, the maps
$(j_{n,n+1})_*$ and $(j_n)_*\colon H_*(M(n))\to H_*(\Minfty)$ are injective.
\end{lemma}
In other words, the elements $\nf_n \in H^*(\Xn)$ are stable under the inclusion $j_{n, n+1}$, and we denote the corresponding element in~$H^*(\Xinfty)$ by~$\nf=[n_0(f_0) \ldots n_{b-1}(f_{b-1})]$. These classes form an additive basis of $H^*(\Xinfty)$.

Consider $\pi_1^* + \cdots + \pi_n^* \; \colon \; H^*(M) \to H^*(M^n)^{S_n} \, \cong \, H^*(M(n))$ which induces a map 
  \[
\pi^* \colon  H^*(M) \to H^*(\Minfty).
  \]
Set 
 \begin{equation} 
p_i \ \coloneqq \ [1(f_i)]= \pi^*(f_i) \ \in \ H^{2s-d_i} (\Minfty).
   \end{equation}
(Note that $p_i$ has degree~$2s-d_i$, not $d_i$ as in the case of the~$u_i$.)
As before, the generator of top degree will play a distinguished role and we give it a special name:
   \begin{equation} \label{def: eta}
\eta \  \coloneqq \ p_0= [1(f_0)] \in H^{2s} (\Xinfty).
   \end{equation}
Let us now show that the classes $p_i$ ($i<b$) generate the cohomology algebra of $\Minfty$ freely.
 \begin{theorem} \label{thm: free cohomology algebra}
The algebra $H^*(\Minfty)$ is free graded-commutative on $p_0, \ldots, p_{b-1}$. Moreover
$p_0,\dots,p_{b-1}$ is a basis of the space of primitive elements of the Hopf algebra $H^*(\Minfty)$.
 \end{theorem}   
 \begin{proof}
We first show that each $p_i=\pi^*(f_i)$ is primitive for $i \in \{ 0, \ldots, b-1\}$. The following diagram is commutative:
\[
     \begin{tikzcd}
M^{n}\times M^{n'} \arrow{r}{\tilde{m}} 
\arrow{d}[swap]{q_n \times q_{n'}} & M^{n+n'} 
\arrow{d}{q_{n+n'}}\\
M(n) \times M(n') \arrow{r}{m}  & M(n+n'),
	\end{tikzcd}
   \]
where $\tilde{m}$ is the canonical isomorphism given by concatenation.  
This implies that the pull-back of the class $m^*( j^*_{n+n'} \,p_i)$ along the quotient map $q_n \times q_{n'}$ equals
\[
\tilde{m}^* \Big( \sum_{k=1}^{n+n'} \pi^*_k(f_i) \Big) \= \sum_{k=1}^{n} \pi^*_k(f_i) \otimes 1 \, + \, \sum_{k=n+1}^{n+n'} 1 \otimes \pi^*_k(f_i),
\]
which is the pullback of $j^*_n \, p_i \otimes 1 + 1 \otimes j^*_{n'} \, p_i $. As $(q_n \times q_{n'})^*$ is injective, passing to the colimit yields $m^*(p_i)=p_i\otimes1+1\otimes p_i$.
 
The Hopf algebra $H^*(\Minfty)$ is connected, graded-commutative and cocommutative of finite type. Since we work over $\Q$, the Milnor--Moore theorem~\cite{milnor1965structure} gives that the inclusion
of the primitives induces an isomorphism
\[
\Sym \big(P H^*(\Minfty) \big)\;\xrightarrow{\ \cong\ }\; H^*(\Minfty).
\]
It therefore suffices to show that $p_0,\dots,p_{b-1}$ is a basis of
$P H^*(\Minfty)$. The classes $p_i$ are linearly independent, since $j^*_1 p_i \; =\; f_i$ and $f_0, \ldots, f_{b-1}$ are part of a basis of $H^*(M)$. 

It remains to verify that the free graded-commutative subalgebra $A\subset H^*(\Minfty)$ generated by $p_0, \ldots, p_{b-1}$ is the whole $H^*(\Minfty)$. Let us compare their dimensions degree
by degree. Both~$A$ and~$H^*(\Minfty)$ are spanned by the bases indexed by the same multisets, namely, $\bp^{\n} = p_0^{n_0} \ldots \, p_{b-1}^{n_{b-1}}$ and $\nf= [n_0(f_0) \ldots n_{b-1}(f_{b-1})]$. The corresponding basis elements have the same degree. Thus the two graded vector spaces have equal dimensions in every degree, and $A=H^*(\Minfty)$, which finishes the proof.
\end{proof}

\begin{proposition} \label{prop: D is a derivation}
Let $\eta$ be the element defined in \eqref{def: eta}. The operation $D$ of capping with $\eta$ is a derivation on the algebra $H_*(\Xinfty)$:
    \[
    D(xy) \ = \ (Dx)  y + x (Dy).
    \]
    Furthermore, capping with the powers $\eta^k$ is the $k$-th iterate of $D$:
    \[
    \eta^k \cap (-) = D^k.
    \]
\end{proposition}
  \begin{proof}
Using the naturality of the cap
product along with the primitivity of $\eta$ we obtain:
 \[
 D(xy) \ = \ \eta \cap (m_*(x \times y))
 \ = \ m_* \big( m^* \eta \cap (x \times y) \big) \ = \ (D x) y + x (Dy).
 \]
 Furthermore, the compatibility formula between cup and cap products yields
 \[
\eta^k \cap \beta \ = \ (\eta^{k-1} \cup \eta) \cap \beta \ = \ \eta^{k-1} \cap (\eta \cap \beta) \ = \ D^ k (\beta).
\qedhere
 \]
\end{proof}

\subsection*{The base-change map $\Phi$} 
As before we denote by $\n$  the vector $(n_0, \ldots, n_{b-1})$. Define the base-change map $\Phi \colon H^*(\Minfty) \to H^*(\Minfty)$ via
\begin{equation} \label{BsChange}
    \Phi(\bp ^{\n}) \; =\; \nf.
\end{equation}
We extend it $\Q$-linearly to obtain a well-defined degree-preserving automorphism of $H^*(\Minfty)$.
\begin{lemma}
      The map $\Phi$ is $\Q[[\eta]]$-linear.
  \end{lemma}
   \begin{proof}
       This follows directly from the fact that 
       \[
       \eta \ [n_0(f_0)\ldots n_{b-1}(f_{b-1})] = [(n_0+1)(f_0)\ldots n_{b-1}(f_{b-1})],
       \]
       which holds since $f_0$ is of top degree in $H^*(M)$, so its cup product with any $f_i \neq 1$ vanishes.
   \end{proof}  
   \subsection*{The universal element $U$} Recall that given any homogeneous basis $\{e_i\}$ of $H^*(M)$ and the dual basis $\{ e^*_i \}$  of $H_*(M)$, the element
 \[
 C= \sum_{i} e_i \otimes e^*_i \ \in \ H^*(M)\otimes H_*(M)
 \]
 does not depend on the choice of basis. This can be seen using the canonical identification $H^*(M)\otimes H_*(M) \cong \Hom(H_*(M), H_*(M))$, under which $C$ corresponds to $\id_{H_*(M)}$. Applying the map $ \id \otimes (j_1)_*$
 to~$C$ we obtain the class
 \begin{equation} \label{eq: def of U}
     U \, \coloneqq  \; \big( \id \otimes (j_1)_* \big) (C) \; =\;  \sum_{i=0}^{b} e_i \otimes u_i \ \in \ H^*(M) \otimes H_*(\Minfty),
  \end{equation}
 which is again independent of the choice of basis.
 \begin{rem}
     We note that the elements of the mixed tensor product $H^*(M) \otimes H_*(\Minfty)$ are graded by total degree and multiplication (cup product in the first component and Pontryagin product in the second) carries the Koszul sign:
 \[
(a \otimes b)(a' \otimes b')=(-1)^{|b|\cdot|a'|} aa' \otimes bb' \ \in \ H^*(M) \otimes H_*(\Minfty).
\]
With this convention $H^*(M) \otimes H_*(\Minfty)$ is a graded-commutative algebra.
\end{rem}
For $x \in H_*(M)$ and $a \otimes b \in H^*(M) \otimes H_*(\Minfty)$ there is an evaluation pairing given by
\[
\langle a \otimes b, x \rangle \; \coloneqq \; (ev_x \otimes \id)(a\otimes b) \=  \langle a, x \rangle \, b\; \in \; H_*(\Minfty).
\]
 
\begin{lemma} \label{lemma: Delta_* is U^r}
    For any positive integer $r$ and $x \in H_*(M)$ we have
    \[
    (\Delta_r)_* (x)\; =\; \langle U^r, x\rangle.
    \]
\end{lemma}
\begin{proof}
    The map $\Delta_r$ can be written as a composition:
    \[
M
\xrightarrow{\ \operatorname{\delta_r}\ }
M^r
\xrightarrow{\ (j_1)^{\times r} \ }
\Minfty^r
\xrightarrow{\ m^{(r)} \ }
\Minfty .
\]
    Therefore, setting $(\delta_r)_*(x)\,=\,\sum_k x^k_{(1)} \otimes\cdots \otimes x^k_{(r)}$, the iterated diagonal coproduct, we have:
    \begin{equation*}
        (\Delta_r)_*(x) \; =\; \sum_k (j_1)_* x^k_{(1)} \ldots (j_1)_* x^k_{(r)}\= \sum_k \bigl\langle U, \, x^k_{(1)} \bigl \rangle \ldots \bigl\langle  U, \, x^k_{(r)} \bigl \rangle \; =\; \langle U^r, x\rangle,  
    \end{equation*}
where we used that $\bigl\langle  U, \, y \bigl \rangle \= (j_1)_* y$ for any $y \in H_*(M)$. In the last equality we used the following multiplicativity for even-degree classes $A, B$:
\begin{equation} \label{coproduct multiplicativity}
    \langle AB, x \rangle \= \sum_k \; \langle A, x^k_{(1)}\rangle \, \langle B, x^k_{(2)} \rangle.
\end{equation}
This is true since for $A=a \otimes a'$ and $B=b \otimes b'$ we have that the Koszul sign $(-1)^{|a'|\cdot|b|}$ from
\[
\langle AB, x \rangle \= (-1)^{|a'|\cdot|b|} \,\langle ab \otimes a'b', x \rangle
\]
combines with the sign of the evaluation $\langle ab, \sum_k x^k_{(1)} \otimes x^k_{(2)} \rangle \= \sum_k (-1)^{|b||x^k_{(1)}|} \langle a, x^k_{(1)} \rangle \langle b, x^k_{(2)} \rangle $. As $\langle a, x^k_{(1)} \rangle$ vanishes unless $|a|=|x^k_{(1)}|,$ the second sign is $(-1)^{|b|\cdot|a|}$. Since by our convention the degree of $A$ is the sum of the degrees of $a$ and $a',$ the total sign is $(-1)^{|b| \cdot |A|}$, which is +1 by the evenness hypothesis for $A$. In our case the class $U$ has even total degree, as each summand $e_i \otimes u_i$ has degree $2d_i$, and hence the multiplicativity \eqref{coproduct multiplicativity} holds if $A$ and $B$ are powers of $U$.
\end{proof}

 \section{Cycle decomposition} \label{Sect3}
\subsection*{Cycle decomposition}
Consider the action of an element $\sigma \in S_n$ on~$M^n$.  We write $\sigma $ as the product of 
disjoint cycles of length $r_1,\dots,r_k\ge1$ with $r_1+\cdots+r_k=n$.  By renumbering the coordinates, we can
write $M^n$ as $M^{r_1}\times\cdots\times M^{r_k}$ and identify $\sigma$ with $\sigma_{r_1}\cdots\sigma_{r_k}$, 
where $\sigma_{r_i}$denotes the cyclic permutation of the coordinates of $M^{r_i}$.   (This is a slight
abuse of notation, since there may be several indices with the same value of $r_i$, but the different factors
$\sigma_{r_i}$ are acting on different copies of $M^{r_i}$ and there should be no confusion.)  In other words,
  \[ 
    (M^n, \sigma) \, =\, \prod_{i=1}^{k} (M^{r_i}, \sigma_{r_i}).
  \]
The fixed-point set $F_i$ of the action of~$\sigma_{r_i}$ on $M^{r_i}$ is simply the diagonal and can be canonically 
identified with $M$ itself. The fixed-point set of $\sigma$ then has a corresponding decomposition 
  \[
  F_\sigma \; = \; \prod_{i=1}^{k} F_i \, = \; \prod_{i=1}^{k} \diag_{r_i}(M) \; \cong \; M^k.
  \]
\begin{lemma} \label{lemma: multiplicativity of L' wrt cycle decomposition}
	The $L'$-class is multiplicative, with respect to the cohomological cross product, relative to the cycle decomposition of $\sigma$:
  \[
  L'(\sigma, M^n) \, = \, L'(\sigma_{r_1}, M^{r_1}) \times\cdots \times L'(\sigma_{r_k}, M^{r_k}).
  \]
\end{lemma}
\begin{proof}
 Let us look at each of the three factors from the definition of $L'(\sigma, M^n)$ (see equation \eqref{eq: three factors of L'}).   
     \begin{itemize}
         \item The tangent bundle $T(F_\sigma)$ decomposes into the external product of the tangent bundles 
          $T(F_i)$, which results in the cross product of $L$-classes:
         \[
         L(F_\sigma)\; =\; L(F_1) \times \cdots \times L(F_k) \ \in \ H^*(F_1\times \cdots \times F_k)=H^*(F_{\sigma}).
         \]
         \item The $(-1)$-eigenbundle of $N^{\sigma}$ decomposes into the external direct sum of the $(-1)$-eigenbundles of 
                the pulled-back normal bundles $N^{\sigma_{ri}} \coloneqq N(F_i \subset M^{r_i} )$ (some of which 
             could be trivial). Both the Euler class and $L^{-1}$ take the external sum to the cross product. 
         \item Again using that the normal bundle $N^\sigma$ decomposes as the external direct sum over $F_\sigma= \prod_i F_i$ 
     of the pulled-back normal bundles $N^{\sigma_{r_i}}$, and that the action of $\sigma$ restricted to each summand is 
         exactly the action of $\sigma_{r_i}$, we get the decomposition for the eigenbundles
         \[N^\sigma_\theta \=  \bigoplus_{i=1}^{k} \, \pr^*_i \, N_{\theta}^{\sigma_{r_i}}, \]
         where $\pr_i $ is the projection of $F_\sigma$ to $F_i$. We note that even though $L_{\theta}$ is unstable (the constant term 
         of its definingcharacteristic power series is different from 1), it is still multiplicative.
        Thus $L_{\theta} (N^{\sigma}_{\theta}) \= \bigtimes_i L_{\theta} (N_{\theta}^{\sigma_{r_i}}).$
     \end{itemize}
     Finally, multiplying the three identities above yields the claim with no extra sign, since all classes that appear have even degrees.
\end{proof}

\subsection*{Multiplicativity of the homological class $l'(\sigma, M^n)$} Note that the action of the symmetric group~$S_n$ on~$M^n$ is orientation-preserving and faithful, so the averaging formula of equation \eqref{eq: averaging formula for L(g, M)} gives 
   \[
 \pi^* L(\Xn) \; =\; \sum_{\sigma \in S_n} L(\sigma, M^n) \ \in  H^*(M^n), 
   \]
where $\pi \colon M^n \to \Xn$ is the projection map. 
Capping both sides with the fundamental class $[M^n]$ and then pushing forward along $\pi \colon M^n \to M(n)$ results in
  \begin{equation} \label{eq: l(X(n)) as a sum of capped L classes}
    n! \: l(M(n)) \; =\; \sum_{\sigma \in S_n} \pi_* \big( L(\sigma, M^n) \cap [M^n] \big).
  \end{equation}
Here for the left-hand side we used that $\deg \pi=|S_n|=n!$ and the naturality of the cap product:
 \begin{align*}
     \pi_* \big(  \pi^* L(\Xn) \cap [M^n] \big) \; &= \; L (\Xn)  \cap  \pi_* [M^n] \\
     &= \; n! \: L(M(n))  \cap [\Xn]\\
     &=\; n! \: l(M(n)).
  \end{align*}
By the construction of the equivariant $L$-class via the Umkehr map \eqref{eq: equivariant L-class def} we have:
   \begin{equation}
       L(\sigma, M^n) \cap [M^n] \ = \ i^\sigma_* (L'(\sigma, M^n) \cap [F_{\sigma}]),
   \end{equation}
where  $i^\sigma$ denotes the inclusion of the fixed-point set $F_\sigma  = (M^n)^\sigma \hookrightarrow M^n$. 
Using our notation 
\[
l'(\sigma, M^n) \; \coloneqq \; L'(\sigma, M^n) \cap [F_{\sigma}] \ \in \ H_*(F_{\sigma}),
\]
we can rewrite equation \eqref{eq: l(X(n)) as a sum of capped L classes} as
 \begin{equation} \label{eq: averaging formula in homology}
     l(M(n))\; =\; \frac{1}{n!} \sum_{\sigma \in S_n} \pi_* i^\sigma_* \big( l'(\sigma, M^n) \big).
 \end{equation}
 \begin{rem}
     In~\cite[Theorem 10.1] {banagl2026equivariant} Banagl established a generalization of the averaging formula~\eqref{eq: averaging formula in homology} for homological $l$-classes of Witt pseudomanifolds.
 \end{rem}
To keep the formulas concise, for the permutation $\sigma_r:M^r\to M^r$ we set 
\[
l'(\sigma_{r}) \ \coloneqq \ l'(\sigma_{r}, M^{r}).
\]
Note that the fixed-point set $F_{\sigma_r}$ of the $\sigma_r$ action on $M^{r}$ is canonically diffeomorphic to $M$. 
The next lemma is a homological counterpart of the multiplicativity of the $l'$-class (now with respect to the homological cross product) relative to the cycle decomposition of $\sigma$.
\begin{lemma} There is an equality of homological $l'$-classes
     \[
  l'(\sigma, M^n) \; = \; l'(\sigma_{r_1}) \times \cdots \times l'(\sigma_{r_k}) \ \in \ H_*(F_{\sigma}).
  \]
Here $\sigma=\sigma_{r_1}\cdots\sigma_{r_k}$ as before and~$\times$ means the homological cross product.
\
\end{lemma}
\begin{proof}
 We have the following equality of fundamental classes
   \[
[M^k]=[F_{\sigma}]=[F_{\sigma_{r_1}}] \times\cdots \times [F_{\sigma_{r_k}}]=\underbrace{[M] \times\cdots \times [M]}_{k}. 
   \]
Using the decomposition for $L'(\sigma, M^n)$ from \autoref{lemma: multiplicativity of L' wrt cycle decomposition} and the formula relating cross and cap products (see e.g.~\cite[5.6.21]{spanier1966homology}) we get
\begin{align*}
     L'(\sigma, M^n) \cap [F_{\sigma}] \; &=\; \big( L'(\sigma_{r_1}, M^{r_1}) \times\cdots \times L'(\sigma_{r_k}, M^{r_k}) \big) \cap  \big( [F_{\sigma_{r_1}}] \times \cdots \times [F_{\sigma_{r_k}}] \big) \\
     &=\; \big( L'(\sigma_{r_1}, M^{r_1}) \cap [F_{\sigma_{r_1}}] \big) \times\cdots \times \big( L'(\sigma_{r_k},M^{r_k}) \cap [F_{\sigma_{r_k}}] \big) \\
     &=\; l'(\sigma_{r_1}) \times\cdots \times l'(\sigma_{r_k}) .
  \end{align*}
 Note that no extra sign appears, since the $L'$-classes live in even degrees.
\end{proof}
\subsection*{Generating series for $l(M(n))$} Let us rewrite our generating series using the multiplicativity with respect to 
the cycle decomposition obtained above. Given a permutation $\sigma \in S_n$ let $(m_1, \ldots, m_n)$ be its cycle type, 
meaning that $\sigma$ has $m_r$ cycles of length $r$. Note that for each tuple $(m_1, \ldots, m_n)$ there are exactly 
$\frac{n!}{\prod_{r=1}^{n} r^{m_r} \, m_r! }$ permutations with this cycle type.  Hence
\begin{equation} \label{Eq: generating series reduction 1}
\begin{aligned} 
    &\sum_{n=0}^{\infty} \; (j_n)_* l(\Xn) \cdot t^n \= \sum_{n=0}^{\infty} \Big( \frac{1}{n!} \ 
\sum_{\sigma \in S_n} (j_n)_{*} \pi_* i^\sigma_* \big( l'(\sigma, M^n) \big) \Big) \cdot t^n\\[5pt]
    &=\; \sum_{n=0}^{\infty} \bigg( \frac{1}{n!} \, \sum_{\substack{m_1, \ldots, m_n \geq 0 \\ 
m_1+2m_2+\cdots\thin=\thin n}} \; \frac{n!}{\prod_{r=1}^{n} r^{m_r} \, m_r!} \ (j_n)_{*} \pi_* i^\sigma_* \Big( l'(\sigma_1)^{\times m_1}
 \times\cdots \times l'(\sigma_n)^{\times m_n} \Big) \cdot t^n \bigg)\\[5pt]
    &=\; \sum_{m_1, m_2, \ldots \geq 0} \ \prod_{r \geq 1} \, \frac{t^{rm_r}}{r^{m_r} \, m_r!} \,
 (j_n)_{*} \pi_* i^\sigma_* \Big( l'(\sigma_1)^{\times m_1} \times l'(\sigma_2)^{\times m_2} \times\cdots \Big)\,,
\end{aligned}
\end{equation}
where $\sigma$ in the last two lines denotes~$\sigma_1^{m_1}\sigma_2^{m_2}\cdots$.  Let us determine the
pushforward of the class 
$l'(\sigma, M^n)=l'(\sigma_1)^{\times m_1} \times\cdots \times l'(\sigma_n)^{\times m_n}$ to~$H_*(\Xinfty)$, induced by the map
  \[
g_{\sigma} \colon  M^k \cong F_{\sigma} \xhookrightarrow{~i^\sigma~} M^n  \xrightarrow{~\pi~} \Xn \xhookrightarrow{~j_n~} \Minfty
  \]
  \[
  g_{\sigma}(x_1,\ldots, x_k)=(\underbrace{x_1, \ldots, x_1}_{r_1}, \ldots, \underbrace{x_k, \ldots, x_k}_{r_k}, *, *, \ldots)\thin,
  \]
where $k$ ($=m_1+m_2+\cdots$) denotes the number of cycles of~$\sigma$ as before.
We note that $g_{\sigma}$ can be written as a composite
  \[
  M^k \; \xrightarrow{\Delta_{r_1} \times\cdots \times \Delta_{r_k}} \; (\Minfty)^k \; \xrightarrow{m^{(k)}} \; \Xinfty,
  \]
where $m^{(k)}$ is the $k$-th iterate of the multiplication map~\eqref{multiplication map on Xinfty}.
Hence, using the naturality of the cross product and the definition of the Pontryagin product, we get
\begin{equation} \label{eq: map g_sigma}
    \begin{aligned}
    (g_{\sigma})_* \big( l'(\sigma, M^n) \big) & \;= \;m^{(k)}_{\, *} \big( (\Delta_{r_1})_*(l'(\sigma_{r_1})) \times\cdots \times (\Delta_{r_k})_*(l'(\sigma_{r_k})) \big)\\
    &\; =\; (\Delta_{r_1})_* (l'(\sigma_{r_1})) \cdot \ldots \cdot (\Delta_{r_k})_*(l'(\sigma_{r_k})).
   \end{aligned}
\end{equation}
Let us denote
  \begin{equation} \label{eq: lambda definition}
   \lambda_r \, \coloneqq \;  (\Delta_{r})_*(l'(\sigma_r)) \ \in \ H_*(\Xinfty).   
  \end{equation}
Equation~\eqref{Eq: generating series reduction 1} combined with~\eqref{eq: map g_sigma} further simplifies the generating series:
\begin{align}
    \sum_{n=0}^{\infty} (j_n)_* l(\Xn) \cdot t^n \ &= 
\sum_{(m_1, m_2, \ldots)} \ \prod_{r \geq 1} \frac{t^{rm_r}}{r^{m_r} \, m_r!} \, \lambda_r^{m_r}\ =\ \prod_{r \geq 1} \bigg( \sum_{m_r \geq 0} \frac{1}{m_r!} \, \Big( \frac{\lambda_r t^r}{r} \Big)^{m_r} \bigg) \nonumber\\[5pt]
&=\ \prod_{r \geq 1} \exp \bigg( {\frac{\lambda_r t^r}{r}} \bigg) \ =\ \exp \bigg( \, {\sum_{r \geq 1}\frac{\lambda_r t^r}{r}} \bigg). \label{eq: final reduction of gen series}
\end{align}
It remains only to compute the classes $\lambda_r$, which we will do in the next section.

\section{Computation for a cycle} \label{Sect4}
The goal of this section is to establish the following computation for a permutation consisting of a single cycle,
which will be the final ingredient in the proof of Theorem~\ref{main theorem: homological formula for l classes}.

\begin{proposition} \label{prop: formula for lambda} The class $\lambda_r$ defined by equation~\eqref{eq: lambda definition} is given by 
  \[
  \lambda_r \ = \ \begin{cases}     \mr \big( l(M) \big) & \text{if r is odd,}\\   e(M) \cdot 1  & \text{if r is even} \end{cases}
  \]
  with $\mu_r$ as in~\eqref{eq: mu_r definition}.
\end{proposition}
For this purpose, we recall the formula for the $L'$-class of a cycle from~\cite{ZagierThesis}. We sketch the proof here for 
completeness while referring to~\cite[Section 9]{ZagierThesis} for full details.
\begin{proposition} \label{lemma: L' of a single cycle}
The $L'$-class of the cyclic permutation on $M^r$ is given by
\[
  L'(\sigma_r, M^r) \ = \ \begin{cases}
  \prod_{j=1}^s\, \frac{x_j}{\tanh(rx_j)} & \text{if r is odd,}\\
  e(TM) & \text{if r is even},
\end{cases}
\]
where the $x_j$ are the formal Pontryagin roots of degree $2$  defined by
 $ p(TM) \ =\ \prod_{j=1}^s\big( 1+x_j^2 \big)$.  
\end{proposition}
\noindent Note that here we have to specify the number of formal roots since the power series $\frac x{\tanh rx}$ 
does not start with~1.  Thus the degree~0 component of $L'(\sigma_r, M^r)$ is $r^{-s}$.
\begin{proof}
The fixed points of the $\sigma_r$-action on $M^r$ are given by the diagonal points, which we denote 
in this section by $\Delta \coloneqq \{(x,\ldots, x) |~x \in M \} \subset M^r$. 
The tangent and the normal bundle of the diagonal submanifold $\Delta$ are given as follows
 \[
T \Delta = TM \otimes \R (1, \ldots, 1),~~ N \Delta = TM \otimes (\R^r/\R(1,\ldots, 1)),
 \]
 where $\sigma_r$ acts by the permutation of coordinates on the second factor.  
 Note that for $r$ odd the action of $\sigma_r$ cannot have $(-1)$ as an eigenvalue. As for the remaining eigenvalues
of the $\sigma_r$-action on $\R^r/\R(1,\ldots, 1)$, they are given by all the non-trivial $r$-th roots of unity, and 
hence the corresponding real blocks have rotation angles given by $\theta_j \coloneqq 2\pi j/r$ for 
$j \in \{1, \ldots, (r-1)/2 \}$. Let $x_j$ be the Chern roots of the complexified tangent bundle of~$M$. Note 
that in this case the Chern roots come in pairs $(x_j, -x_j)$ and each $N_{\theta_i}$ can be viewed as a complexification of~$TM$.
 Hence, using the defining equation~\eqref{eq: three factors of L'} for~$L'$, we have that for $r$ odd 
\[
  L'(\sigma_r, M^r) \ = \ (-1)^{s(r-1)/2} \ L(\Delta) \prod_{k=1}^{(r-1)/2} L_{\theta_k}(N_{\theta_k}), 
 \]
where the sign accounts for the difference between the orientation of the fixed-point set $\Delta$ coming from
the Atiyah-Singer recipe and the one obtained by identifying $\Delta$ with~$M$. Therefore,
  \begin{align}
  L'(\sigma_r, M^r) \ \nonumber
  &= \ L(TM) \prod_{k=1}^{(r-1)/2} \prod_{j=1}^{s} \Bigl(-\thin\coth(x_j+i\pi k/r) \, \coth(-x_j+i\pi k/r)\Bigr)\\ \nonumber
  &= \ L(TM) \prod_{j=1}^{s} \bigg( \prod_{k=1}^{r-1} \coth(x_j +i\pi k/r) \bigg) \nonumber \\ 
  &= \ \bigg( \prod_{j=1}^s \frac{x_j}{\tanh(x_j)} \bigg) \bigg( \prod_{j=1}^{s} \frac{\coth(rx_j)}{\coth(x_j)} \bigg) \ = \  \prod_{j=1}^s \frac{x_j}{\tanh(rx_j)}\, ,   \label{eq: formula for product of coth}
  \end{align}
where in the passage to equation~\eqref{eq: formula for product of coth} we used the standard identity
   \[
  \prod_{k=0}^{r-1} \coth(z + \frac{i \pi k}{r}) \ = \ \begin{cases}
      \coth(rz) &\textup{if $r$ is odd,}\\
      1 & \textup{if $r$ is even}
  \end{cases}
   \]
for the hyperbolic cotangent and the expression for the $L$-class of the tangent bundle $L(TM)=\prod_{j=1}^s \frac{x_j}{\tanh(x_j)}$.

If $r$ is even, then there is also a $(-1)$-eigenspace in the normal bundle decomposition, contributing the $\pi$-factor
  \[
  L_{\pi}(N_{\pi})\ = \ (-1)^{s(r/2-1)} \; e(TM) \; L(TM)^{-1},
  \]
which cancels the $L(\Delta)$ factor in the defining equation~\eqref{eq: three factors of L'}, leaving
 \[
  L'(\sigma_r, M^r) \ = \ (-1)^{s(r/2-1)} \; e(TM) \prod_{k=1}^{r/2-1} L_{\theta_k}(N_{\theta_k}) \ = \ e(TM). \qedhere
 \]
\end{proof}

\begin{proof}[Proof of Proposition~\ref{prop: formula for lambda}]
Recall that 

\[
\lambda_r \; = \; (\Delta_{r})_* \big(l'(\sigma_r) \big)=(\Delta_{r})_* \big( L'(\sigma_r, M^r) \cap [M] \big)  \ \in \ H_*(\Minfty). 
\]
By Proposition~\ref{lemma: L' of a single cycle}, for $r$ odd
\[
L'(\sigma_r, M^r)\; = \; \prod_{j=1}^{s} \frac{x_j}{\tanh(rx_j)} \; =\; r^{-s} \, \prod_{j=1}^{s} \frac{rx_j}{\tanh(rx_j)} \; =\; r^{-s} \, \psi_r(L(M)),
\] 
where $\psi_r \colon H^{\ev}(M) \to H^{\ev}(M)$ is a degree-preserving map that rescales the elements in degree $2j$ by $r^j$. 
Let us look at the degree $2j$ component of $\lambda_r$:
  \[
 \big( r^{-s} \, \psi_r (L(M)) \cap [M] \big)_{2j} \; = \; r^{-s} \, r^{s-j} \, L_{2(s-j)}(M) \cap [M]\; =\;r^{-j} \, l(M)_{2j},  
 \]
 which is exactly the desired rescaling from the definition~\eqref{eq: mu_r definition} of $\mr$. Hence, for odd $r$, we have
 \[
 \lambda_r \; =\; \mr(l(M)).
 \]
 For even $r$, $\lambda_r \, =\, (\Delta_r)_* \big( e(TM) \cap [M] \big) \,= \, e(M) \cdot 1$, which finishes the proof.
\end{proof}
Finally, we substitute the formula of Proposition~\ref{prop: formula for lambda} into the expression for the generating series~\eqref{eq: final reduction of gen series} and split by parity, using 
\[
\sum_{r~\textup{even}} \frac{t^r}{r} \; =\; -\frac{1}{2} \log(1-t^2),
\]
to obtain Theorem~\ref{main theorem: homological formula for l classes} as stated in the introduction.

\section{From homology back to cohomology} \label{Sect5}

The goal of this section is to give a short proof of Theorem~\ref{thmB} using the generating series for homological $l$-classes formula from Theorem~\ref{main theorem: homological formula for l classes}. 

Recall that the capping operation $D=\eta \cap (-)$ is a derivation by Proposition~\ref{prop: D is a derivation}, where  $\eta \in H^{2s}(\Minfty)$ is the special class defined in~\eqref{def: eta}. We also denote by the same letter $D$ its coefficientwise extension to the ring of formal power series $H_*(\Xinfty)[[t]]$. Denote by $\E$ the right-hand side of the generating series for $l$-classes in Theorem~\ref{main theorem: homological formula for l classes}:
 \begin{equation} \label{eq: def of E}
 \E \ \coloneqq \ \Big( \frac{1}{1-t^2}\Big) ^{e(M)/2}  e^{\Psi(l(M))} \ \in \ H_*(\Minfty)[[t]].
 \end{equation} 
\begin{proposition} \label{prop: diff equation for the RHS}
The element $\E$ satisfies the first order differential equation
    \[
    D(\E) \= g_s(t)  \, \E,
    \]
where $g_s(t)$ is the power series defined in~\eqref{Defgs}.    
\end{proposition}
\begin{proof} Rewrite $\E$ as $\E \ = \ e^{ \mathcal{C} + \T }$, with $\mathcal{C}=-\frac{e(M)}{2}\log(1-t^2)$ and $\T= \Psi(l(M))$ and observe that~$\mathcal{C}$ has homological degree zero. Since $D$ lowers the degree by $2s$, we have
	\[D (\mathcal{C})=0.\]
On the other hand, the pullback of $\eta$ to $H^*(M^r)$ is $\sum_{i=1}^{r} \pi^*_i (f_0)$, where $\pi_i \colon M^r \to M$ denotes the projection onto the $i$-th factor. Hence
  \begin{equation} \label{eq: delta of eta}
\Delta_r^* (\eta) \= \Delta_r^* \Big( \sum_{i=1}^{r} \pi^*_i (f_0) \Big)=r f_0.
  \end{equation}
For odd $r$, again using that $D$ lowers the degree by $2s$, we note that only the top component of $\mu_r  (l(M))$, namely $\mu_r [M]=r^{-s} (\Delta_r)_*[M]$,  survives under $D$. Next, using the naturality of the cap product, $\eta \cap (\Delta_r)_* \gamma = (\Delta_r)_* (\Delta^*_r \eta \cap \gamma)$, together with equation~\eqref{eq: delta of eta}, we get
 \begin{align*}
D \mu_r(l(M)) \ &= \ r^{-s} D \big( (\Delta_r)_* [M] \big)\
= \ r^{-s}(\Delta_r)_* \big( \Delta_r^* \eta \cap [M] \big)\ \\
&= \ r^{-s} (\Delta_r)_* \big(rf_0 \cap[M] \big)\
= \ r^{1-s}.
  \end{align*}   
Hence
   \[
D (\T) \= D  \Big( \sum_{\substack{r\geq 1\\ r  \mathrm{~odd}}} \mu_r(l(M)) \cdot \frac{t^r}{r} \Big) \= \sum_{\substack{r\geq 1\\ r  \mathrm{~odd}}} \frac{t^r}{r^s} \= g_s(t),
   \]
and therefore
   \[
  D(\E) \= \E \cdot D(\T) \= g_s(t) \, \E. \qedhere
   \]
\end{proof}
By Proposition~\ref{prop: diff equation for the RHS} we have 
\[
D^k (\E) \= g_s(t)^k \, \E
\]
for all $k \geq 1$, since the homological degree of $g_s(t)$ is zero. Recalling that $f_s(t)$ is the inverse power series 
of $g_s(t)$, we obtain
  \begin{equation} \label{eq: final dif equation}
f_s(\eta) \cap \E \= f_s(D) \, \E \ =\ f_s(g_s(t)) \, \E \ = \ t \, \E.
  \end{equation}
Let us cap both sides of the formula in Theorem~\ref{main theorem: homological formula for l classes} with~$f_s(\eta)$ 
and compare the coefficients of~$t^{n+1}$. On the left-hand side, the capping being coefficientwise, this coefficient is 
$f_s(\eta) \cap (j_{n+1})_*l(M(n+1))$. At the same time, by equation~\eqref{eq: final dif equation} it is also the coefficient
of~$t^n$ in the original generating series (before capping), hence
  \[
  f_s(\eta) \cap (j_{n+1})_* \, l(M(n+1)) \= (j_n)_* \, l(M(n)).
  \]
  By the naturality of the cap product and the injectivity of $(j_{n+1})_*$, this can be rewritten as an equation in~$H_*(M(n+1))$:
  \begin{equation} \label{eq: consecutive homological l-classes}
f_s(\eta_{n+1}) \cap l(M(n+1)) \= j_{n+1}^* f_s(\eta) \cap l(M(n+1)) \ =\ j_* l(M(n)),
  \end{equation}
  where $\eta_{n+1} \coloneqq j_{n+1}^* \eta$ and $j \coloneqq j_{n, n+1} \colon M(n) \to M(n+1)$ is the inclusion.
In particular, restricting equation~\eqref{eq: consecutive homological l-classes} to its top degree gives
  \begin{equation}
j_* [M(n)] \= \eta_{n+1} \cap [M(n+1)].
  \end{equation}
 We are now ready to conclude with our short proof of the cohomological formula from the earlier 
paper~\cite{ZagierThesis} (Section 8, Theorem 1). 
 \begin{thmB'} Let $M$ be a manifold of dimension $2s$. The cohomological $L$-classes of consecutive symmetric products of $M$ are related by the equation
     \begin{equation} \label{eqB'}
     j^* L( M(n+1)) \=  L(M(n))  \ j_n^* \, Q_s(\eta),
     \end{equation}
  where $j \coloneqq j_{n, n+1} \colon M(n) \to M(n+1)$ and $j_n \colon M(n) \to \Minfty$ are the inclusions.   
 \end{thmB'}
 \begin{proof}
  Let us cap each side with $[M(n)]$ and push forward along $j_*$ : 
 \begin{align*}
  j_* \big( j^* L(M(n+1)) \cap [M(n)] \big)\
  &= \ L(M(n+1)) \cap \big( \eta_{n+1} \cap [M(n+1)] \big)\\
  &=\ \eta_{n+1}  \cap \, l(M(n+1)),
   \end{align*}
   \begin{align*}
  j_* \big( L(M(n)) \ j_n^* Q_s(\eta) \cap [M(n)] \big) \ &=\  j_* \big( j_n^* Q_s(\eta) \cap l(M(n)) \big)\\
  &= \ j_{n+1}^* Q_s(\eta) \cap j_* l(M(n)) \\
   &= \ Q_s(\eta_{n+1}) \cap \big( f_s(\eta_{n+1}) \cap l(M(n+1)) \big)\\
     &=\eta_{n+1} \cap l(M(n+1)).
  \end{align*}
Since $j_*$ is injective and capping with $[M(n)]$ is an isomorphism, we conclude that
\[
j^* L( M(n+1))\ = \ L(M(n))  \ j_n^* Q_s(\eta). \qedhere
\]

 \end{proof}
 \begin{cor}[= Theorem B] \label{cor: eq L=GQ^n}
   There is an element $G \in H^*(\Xinfty)$ independent of $n$ such that 
  \begin{equation} \label{Eq: L=GQ^n}
    L(M(n))=j_n^* \big( G \, Q_s(\eta)^{n+1} \big).
   \end{equation}
 \end{cor}
 
\begin{rem}
Equation~\eqref{eqB'} is reminiscent of the usual formula $L(\nu) L(A)=j^* L(B)$ for the normal bundle~$\nu$
of a submanifold $j\colon A \to B.$  This suggests that there should exist some notion of ``$s$-rational homology 
bundles'' over topological spaces, locally modeled in some way on inclusions of symmetric products of $(2s)$-dimensional
manifolds.  These bundles would have some kind of ``$s$-Pontryagin class" and  $L$-classes  given by the
multiplicative sequence associated with the power series $Q_s(t)$ instead of the usual $t/\tanh(t)$. This idea was 
already mentioned in the final paragraph of~\cite{zagier1972}, but we still do not know a reasonable geometric
definition. We hope to return to this later.
\end{rem}

\section{Closed formula for the factor $G$} \label{Sect6}
In this section we derive the closed formula for the factor $G$ from equation~\eqref{Eq: L=GQ^n} using the generating 
series formula for $l(M(n))$ of Theorem~\ref{main theorem: homological formula for l classes}. Let us rewrite the
 universal element $U$ defined in~\eqref{eq: def of U} separating out the two special summands, $i=0$, contributing $1$,
 and $i=b$, contributing $e_b \otimes \alpha$:
\begin{equation} \label{eq: def U truncated}
     U\=1 + e_b \otimes \alpha + \mathcal{U}, \qquad \mathcal{U} \coloneqq \sum_{i=1}^{b-1} e_i \otimes u_i.
 \end{equation}

 \begin{proposition} \label{prop: mu(l(X)) with alpha term separated}
 	Let $\Psi$ be the map defined in~\eqref{PsiSer}, which encodes the twists by
 	the Adams operations. Then
     \[
     \Psi\bigl(l(M)\bigr) \= \ g_s(t)\thin\alpha \+\sum_{i=0}^{s} \bigl\langle g_{i+1}\big(t(1 +\mathcal{U}) \big), l_i\bigr\rangle
     \]
 \end{proposition}
 \begin{proof}
     Note that $e_b \cup e_i=0$ for all $i \neq 0$ since $e_b$ is a top degree class of $H^*(M).$ Thus
     \[
     U^r=(e_b \otimes \alpha + (1+\U))^r \ = \ r e_b \otimes \alpha+(1+\U)^r.
     \]
  By \autoref{lemma: Delta_* is U^r}
  \begin{equation} \label{eq: expanding Delta}
  (\Delta_r)_*(l_i)=\langle r e_b \otimes \alpha + (1+\U)^r, l_i \rangle =r\alpha \delta_{s,i} + \langle(1+\U)^r, l_i \rangle.
  \end{equation}
 Multiplying~\eqref{eq: expanding Delta} by $t^r/r^{i+1}$ and summing over all $i\ge0$ and all odd $r$ gives the desired identity:
  \[
  \Psi(l(M))=\sum_{\substack{r\geq 1\\ r  \mathrm{~odd}}}\frac{t \alpha}{r^s} +\sum_{i=0}^{s} \sum_{\substack{r\geq 1\\ r  \mathrm{~odd}}} \frac{t^r}{r^{i+1}}\langle(1+\U)^r, l_i \rangle=g_s(t)\alpha+\sum_{i=0}^{s} \langle g_{i+1}\big(t(1 +\mathcal{U}) \big), l_i\rangle. \qedhere
  \]
 \end{proof}
Recall that the classes $u_i \in H_{d_i}(\Minfty)$ are defined by
\[
u_i \, \coloneqq \; (j_1)_*(e^*_i) \; =\; (j_1)_*(f_i \cap [M]),
\]
where the last equality follows from our definition of $f_i$'s as the cohomological basis Poincar{\'e} dual to the $e_i$'s. Recall also that the classes $p_i \= [1(f_i)] \ \in \ H^{2s-d_i} (\Minfty)$ are free (graded-commutative) generators of the algebra $H^*(\Minfty)$.
\begin{lemma} \label{lemma: Psi cap e^xa}
    For any power series $F=F(p_0, \ldots, p_{b-1})$ there is an identity
    \[
    \sum_{n \geq 0} \big( \Phi(F) \cap \frac{\alpha^n}{n!} \big) x^n=e^{x\alpha} F(x, xu_1, \ldots , xu_{b-1}),
    \]
    where $\Phi$ is the base-change map defined in~\eqref{BsChange}.
\end{lemma}
  \begin{proof}
      It is enough to check the formula for $F$ a monomial $p_0^{n_0} p_1^{n_1} \ldots p_{b-1}^{n_{b-1}}$.
      We have 
      \[
      \Phi(F) \cap \frac{\alpha^n}{n!}=[n_0(f_0)\ldots n_{b-1}(f_{b-1})] \cap \frac{\alpha^n}{n!}
      =\frac{\alpha^{n -(n_0 + \ldots + n_{b-1})}}{( n-(n_0 + \ldots n_{b-1}) )!} u_1^{n_1}  \ldots  u_{b-1}^{n_{b-1}}
      \]
      if $n_0 + \cdots +n_{b-1} \leq n$ and zero otherwise.
Hence
\begin{align*}
  \sum_{n \geq 0} \big( \nf \cap \frac{\alpha^n}{n!} \big) x^n &= \sum_{n \geq n_0 +\ldots + n_{b-1}} \Big( \frac{\alpha^{n -(n_0 + \ldots + n_{b-1})}}{( n-(n_0 + \ldots n_{b-1}) )!} u_1^{n_1}  \ldots  u_{b-1}^{n_{b-1}} \Big) x^n\\
  &=\sum_{n \geq n_0 +\ldots + n_{b-1}}  \frac{(x\alpha)^{n -(n_0 + \ldots + n_{b-1})}}{( n-(n_0 + \ldots n_{b-1}) )!} \ (x)^{n_0} (xu_1)^{n_1} \ldots  (xu_{b-1})^{n_{b-1}}\\
  \\
  &= e^{x\alpha}F(x,xu_1, \ldots, xu_{b-1} ). \qedhere
\end{align*}

  \end{proof}
  \begin{rem}
       The classes $u_1, \ldots, u_{b-1}, \alpha=u_b$ are free graded-commutative generators of the algebra $H_*(\Minfty)$ by \autoref{thm: free hopf algebra}, and likewise $\eta=p_0,p_1, \ldots, p_{b-1}$ are free graded-commutative generators of the algebra $H^*(\Minfty)$ by \autoref{thm: free cohomology algebra}. The left-hand side of the lemma determines the image of $F$ under the
substitution $p_0 \mapsto x$, $p_i \mapsto xu_i$
$(1\leq i\leq b-1)$, with
\[
  p_0^{n_0}p_1^{n_1}\cdots p_{b-1}^{n_{b-1}}
  \;\longmapsto\;
  x^{\,n_0+n_1+\dots+n_{b-1}}\,u_1^{n_1}\cdots u_{b-1}^{n_{b-1}} 
\]
which is clearly injective and hence determines the series $F=F(p_0, \ldots, p_{b-1})$ uniquely. 
 \end{rem}
Let 
\[
\p \coloneqq \sum_{i=1}^{b-1} e_i \otimes p_i \in \ H^*(M)\otimes H^*(\Minfty)
\]
be the cohomological analogue of the truncated universal element $\mathcal{U}$ defined in~\eqref{eq: def U truncated}. We now show the closed formula for the $G$-class which, together with Proposition~\ref{prop: explicit formula for H} below, constitutes the content of Theorem~\ref{thmC}.
\begin{theorem} \label{thm: main formula for G}
Let $f_s$ be the inverse power series of the odd polylogarithm series $g_s$ defined in~\eqref{Defgs}, let $l_i \in H_{2i}(M)$ denote the $i$-th component of the homological $l$-class of $M$, and set	
\[
H\, =\, \sum_{i=0}^{s} \Big\langle g_{i+1} \Big(\frac{f_s(\eta)}{\eta}(\eta + \p) \Big) , \, l_i \Big\rangle 
\ \in \ \Q[[\eta]][p_1, \ldots, p_{b-1}]. \]	
Then the class $G$ from equation~\eqref{Eq: L=GQ^n} is given by
\[
    G \,=\, \frac{f'_s(\eta)}{(1-f_s(\eta)^2)^{e(M)/2}} \; \Phi(e^H).
\]
\end{theorem}
\begin{proof}
    Since $\Phi$ is a bijection we can choose $F$ in such a way that $G=\Phi(F)$. By Theorem~\ref{thmB}
    \begin{equation} \label{eq: L(X(n)) in theorem}
    L(M(n)) \; =\; j_n^* \big( G \, Q_s(\eta)^{n+1} \big).
   \end{equation}
   Capping equation~\eqref{eq: L(X(n)) in theorem} with $[M(n)]$ and pushing it forward along $(j_n)_*$ gives
   \[
   (j_n)_*l(M(n))= G Q_s(\eta)^{n+1} \cap (j_n)_*[M(n)]=G Q_s(\eta)^{n+1} \cap \frac{\alpha^n}{n!}
   \]
where the last equality is the result of \autoref{lem: [X(n)] in alpha}. From Theorem~\ref{main theorem: homological formula for l classes}, comparing the coefficients at $t^{n}$ and using that $G=\Phi(F)$ along with the linearity of $\Phi$ with respect to power series in $\eta$, we obtain
\begin{align*}
    [t^n] \, \E& \; =\; (Q_s^{n+1}(\eta) \, G) \cap \frac{\alpha^n}{n!} \; =\;\Phi \big( Q_s^{n+1}(\eta) \, F \big) \cap \frac{\alpha^n}{n!} \\
    &\; = \; [x^n] \, \big( e^{x\alpha} \, Q_s^{n+1}(x) \, F (x, xu_1, \ldots, xu_{b-1}) \big),
\end{align*}
where the last equality follows from \autoref{lemma: Psi cap e^xa}.
Using that $Q_s(x)=x/f_s(x)$ and the Lagrange inversion formula (or residue calculus) for the change of variable $x=g_s(z)$, we get
\begin{align*}
[x^n] \big( e^{x\alpha} F (x, xu_1, \ldots, xu_{b-1}) \, Q_s^{n+1}(x)\big) 
&=\, \res_{x=0} \; \Bigl(\frac{e^{x\alpha} F (x, xu_1, \ldots, xu_{b-1}) }{f_s(x)^{n+1}}\, dx\Bigr)\\
&=\, [z^n] \, g'_s(z) \, e^{g_s(z)\alpha} \, F \big(g_s(z),g_s(z)u_1, \ldots, g_s(z)u_{b-1} \big).
\end{align*}
Therefore, 
\[
\E(t)\= g'_s(t) \; e^{g_s(t)\alpha} \; F(g_s(t),g_s(t)u_1, \ldots, g_s(t)u_{b-1}),
\] 
and substituting $t=f_s(x)$ yields
\[
F(x, xu_1, \ldots, xu_{b-1}) \= f'_s(x) e^{-x \alpha} \, \E\big(f_s(x) \big).
\]
Using the defining equation~\eqref{eq: def of E}  for $\E$ and applying Proposition~\ref{prop: mu(l(X)) with alpha term separated} cancels out the $e^{-x \alpha}$ factor, resulting in:
\[
F(x, xu_1, \ldots, xu_{b-1}) \,=\, \frac{f'_s(x)}{\big( 1-f_s(x)^2 \big)^{e(M)/2}} \ \exp \Big( \,\sum_{i=0}^{s}\langle g_{i+1}(f_s(x)(1 +\mathcal{U}) ), \, l_i\rangle \Big).
\]
Finally, to retrieve $F=F(p_0, \ldots, p_{b-1})$, we note that under the substitution $\eta=p_0 \mapsto x$, $p_i \mapsto xu_i$ $(1\leq i\leq b-1)$, the preimage of the class
\[
\Big\langle g_{i+1}\big(f_s(x)(1 +\mathcal{U}) \big), \, l_i\Big\rangle \; = \; \Big\langle g_{i+1}\Big(\frac{f_s(x)}{x}(x +x\mathcal{U}) \Big), \, l_i\Big\rangle
\]
is $\big\langle g_{i+1} \big(\frac{f_s(\eta)}{\eta}(\eta + \p) \big) , \, l_i \big\rangle$, which is exactly the corresponding summand of $H$. Hence
\[
G\; =\; \Phi(F) \; =\; \frac{f'_s(\eta)}{\big(1-f_s(\eta)^2 \big)^{e(M)/2}} \; \Phi(e^{H}). \qedhere
\]
\end{proof}
Finally, we provide a more explicit expression for the element~$H$ appearing in \autoref{thm: main formula for G}. For $x \in H_*(M)$ set 
\begin{equation}
    \varepsilon^k(x) \, \coloneqq \, \sum_{1 \leq i_1, \ldots, i_k \leq b-1} \langle e_{i_1} \ldots e_{i_k}, x \rangle p_{i_k} \ldots p_{i_1} \ \in \; H^*(\Minfty).
\end{equation}
For $k=0$ we set $\varepsilon^0(x) \, \coloneqq \, \langle 1, x\rangle.$
The next lemma explains the reversed order of the indexing for the $p$-classes: it absorbs the Koszul sign. 
\begin{lemma} \label{P^k is mu^k}
  For every positive integer $k$ and $x \in H_*(M)$ homogeneous of even degree we have: 
\[
\varepsilon^k(x)  \=  \langle \, \p^k, \,x \,\rangle \,.
\]
\end{lemma}
\begin{proof}
Iterating the Koszul rule gives
\[
\p^k \= \sum_{1 \leq i_1, \ldots, i_k \leq b-1} (-1)^{\sum_{a<b} |e_{i_a}| |p_{i_b}|} \; e_{i_1} \ldots e_{i_k} \otimes p_{i_1} \ldots p_{i_k}.
\]
Modulo 2 we have $\deg p_i =2s-d_i \equiv d_i=\deg e_i$, so the exponent of the sign is $\sum_{a<b} d_{i_a} d_{i_b},$ the sum of the pairwise products of the degrees $d_{i_1}, \ldots, d_{i_k}$. On the other hand, this is also the sign we get when reversing the order in the product $p_{i_1} \ldots p_{i_k}$. Therefore, the reversed order of the $p$'s cancels out the sign, as claimed. 
\end{proof}
\begin{proposition} \label{prop: explicit formula for H} The class $H$ of Theorem~\ref{thm: main formula for G} is given by
    \begin{equation} 
        H\=\sum_{i=0}^{s} \sum_{k=0}^{2i} \frac{1}{k!}\Big(\frac{f_s(\eta)}{\eta} \Big)^k  g^{(k)}_{i+1}(f_s(\eta)) \, \varepsilon^k(l_i),
     \end{equation}
     where $g_{i+1}^{(k)}$ denotes the $k$-th derivative of the odd polylogarithm $g_{i+1}$.
\end{proposition}
\begin{proof}

Consider the (formal) Taylor expansion of $g_{i+1}$ at $f_s(\eta)$ and use \autoref{P^k is mu^k}:
\begin{align*}
    \Big\langle \, g_{i+1} \Big(\frac{f_s(\eta)}{\eta}(\eta + \p) \Big) , \, l_i \, \Big\rangle \, &=\,   \Big\langle \sum_{k \geq 0} \frac{1}{k!} \, g_{i+1}^{(k)} \big( f_s(\eta) \big) \, \Big( \frac{f_s(\eta)}{\eta} \Big)^k \; \p^k , \, l_i \Big\rangle\\
    &= \, \sum_{k=0}^{2i} \frac{1}{k!} \,\Big( \frac{f_s(\eta)}{\eta} \Big)^k \, g_{i+1}^{(k)} \big( f_s(\eta) \big) \, \varepsilon^k(l_i).
\end{align*}
The truncation at $k=2i=\deg(l_i)$ occurs since every $e_j \in H^*(M)$ with $1 \leq j \leq b-1$ has positive degree. 
Summing over $i$ yields the formula for~$H$ and completes the proof of Theorem~\ref{thmC}.
\end{proof}

\setlength{\parskip}{\storeparskip}


\begin{thebibliography}{99}
	
	\bibitem{adams1963vector}
	J. F. Adams, \emph{Vector fields on spheres}, Ann. of Math. (2),~{\bf 75} (1962), 603--632.
	
	\bibitem{banagl2026equivariant}
	M. Banagl, \emph{Equivariant $L$-classes of Atiyah--Singer--Zagier type for singular spaces}, Bulletin des Sciences
	Mathématiques (2026): 103885.
	
	\bibitem{cappell2017characteristic}
	S. E. Cappell, L. Maxim, J. Sch\"urmann, J. L. Shaneson, and S. Yokura, \emph{Characteristic classes of symmetric products of complex quasi-projective varieties}, J. Reine Angew. Math.,~{\bf 728} (2017), 35--63.
	
	\bibitem{cartier2007primer}
	P. Cartier, \emph{A primer of Hopf algebras}, Frontiers in Number Theory, Physics, and Geometry~II, Springer, Berlin (2007), 537--615.
	
	\bibitem{hirzebruch1966topological}
	F. Hirzebruch, \emph{Topological methods in algebraic geometry}, Vol. 175. Springer Berlin-Heidelberg-New York (1966).
	
	\bibitem{HirzGenSer}
	F. Hirzebruch, \emph{Lectures on the Atiyah-Singer theorem and its applications}, Berkeley, 
Summer 1968 (unpublished lecture notes).
	
	\bibitem{MacCurves}
	I. G. Macdonald, \emph{Symmetric products of an algebraic curve}, Topology,~{\bf 1}(4), (1962), 319--343.
	
	\bibitem{milnor1965structure}
	J. W. Milnor and J. C. Moore, \emph{On the structure of Hopf algebras}, Ann. of Math. (2),~{\bf 81} (1965), 211--264.
	
	\bibitem{milnor1974characteristic}
	J. W. Milnor and J. D. Stasheff, \emph{Characteristic classes}, Princeton Univ. Press (1957).
	
	\bibitem{serre1953groupes}
	J. P. Serre, \emph{Groupes d'homotopie et classes de groupes ab\'eliens}, Ann. of Math. (2),~{\bf 58} (1953), 258--294.
	
	\bibitem{spanier1966homology}
	E. H. Spanier, \emph{Algebraic topology}, McGraw-Hill, New York (1966).
	
	\bibitem{thom1958}
	R. Thom, \emph{Les classes caract\'eristiques de Pontrjagin des vari\'et\'es triangul\'ees}, Symp. Intern. Top. Alg., 
Universidad Nacional Aut\'onoma de M\'exico and UNESCO, Mexico City (1958), 54--67.
	
	\bibitem{zagier1972}
	D. Zagier, \emph{The Pontrjagin class of an orbit space}, Topology,~{\bf 11}(3), (1972), 253--264.
	
	\bibitem{ZagierThesis}
	D. Zagier, \emph{Equivariant Pontrjagin classes and applications to orbit spaces}, Lecture Notes in Math.~{\bf 290}, Springer, Berlin (1972).
	
	
	
\end{thebibliography}
\end{document}